\documentclass{llncs}
\usepackage{amssymb}
\usepackage{graphicx,color}
\newcommand{\ignore}[1]{}
\newcommand{\boxi}{\ensuremath{\mathrm{box}}}
\newcommand{\bbox}{\rule{0.6em}{0.6em}}

\begin{document}

\title{Boxicity of Leaf Powers}
\author{L. Sunil Chandran\and Mathew C. Francis\and Rogers Mathew}
\institute{Department of Computer Science and Automation, Indian Institute of
Science, Bangalore -- 560012, India.\\
\texttt{\{sunil,mathew,rogers\}@csa.iisc.ernet.in}}
\maketitle
\bibliographystyle{plain}
\begin{abstract}
The boxicity of a graph $G$, denoted as $\boxi(G)$ is defined as the minimum
integer $t$ such that $G$ is an intersection graph of axis-parallel
$t$-dimensional boxes. A graph $G$ is a $k$-leaf power if there exists
a tree $T$ such that the leaves of the tree correspond to the vertices
of $G$ and two vertices in $G$ are adjacent if and only if their corresponding
leaves in $T$ are at a distance of at most $k$. Leaf powers are a subclass
of strongly chordal graphs and are used in the construction of phylogenetic trees
in evolutionary biology. We show that for a $k$-leaf power $G$, $\boxi(G)\leq
k-1$. We also show the tightness of this bound by constructing a $k$-leaf power
with boxicity equal to $k-1$. This result implies that there exists strongly
chordal graphs with arbitrarily high boxicity which is somewhat counterintuitive.

\end{abstract}
\section{Introduction}
An \emph{axis-parallel $k$-dimesional box}, or \emph{$k$-box} in short,
is the Cartesian product $R_1\times R_2\times\cdots\times R_k$ where each
$R_i$ is an interval of the form $[a_i,b_i]$ on the real line.
A 1-box is thus just a closed interval on the real line and a 2-box a
rectangle in $\mathbb{R}^2$ with its sides parallel to the axes.
A graph $G(V,E)$ is said to be an \emph{intersection graph} of $k$-boxes
if there is a mapping $f$ that maps the vertices of $G$ to $k$-boxes
such that for any two vertices $u,v\in V$, $(u,v)\in E(G)\Leftrightarrow
f(u)\cap f(v)\not=\emptyset$. Then, $f$ is called a \emph{$k$-box
representation} of $G$.
Thus interval graphs are exactly the intersection graphs of
1-boxes. Clearly, a graph that is an intersection graph of $k$-boxes is
also an intersection graph of $j$ boxes for any $j\geq k$.
The \emph{boxicity} of a graph $G$, denoted as $\boxi(G)$,
is the minimum integer $k$ such that $G$ is an intersection graph of $k$-boxes.

Roberts\cite{Roberts} gave an upper bound of $n/2$ for the boxicity of any
graph on $n$ vertices and showed that the complete $n/2$-partite graph with 2
vertices in each part achieves this boxicity. Boxicity has also been shown
to have upper bounds in terms of other graph parameters such as the maximum
degree and the treewidth\cite{CN05}.
It was shown in \cite{boxdellogn} that for any graph $G$ on $n$ vertices and
having maximum
degree $\Delta$, $\boxi(G)\leq \lceil(\Delta+2)\ln n\rceil$. The same authors showed
in \cite{CFNMaxdeg} that $\boxi(G)\leq 2\Delta^2$.
This result shows that the boxicity of any
graph with bounded degree is bounded no matter how large the graph is.

The boxicity of several special classes of graphs have also been studied.
Scheinerman \cite{Scheiner} showed that outerplanar graphs have boxicity at
most 2 while Thomassen \cite{Thoma1} showed that every planar graph has
boxicity at most 3.
The boxicity of series-parallel graphs was studied in \cite{CRB1} and that of
Halin graphs in \cite{halinbox}.
\ignore{
	 An interesting class of graphs for which the
	question of
	boxicity is still open is the class of chordal bipartite graphs. It was shown
	in cit that the boxicity of $P_6$-free chordal bipartite graphs is at most 2.
	The authors conjecture that the boxicity of chordal bipartite graphs is at
	most 2. In fact it is not known whether the boxicity of chordal bipartite
	graphs can be bounded by a constant.
	An approach to the problem might be to complete one partition of
	the chordal bipartite graph whereby on ends up with a ``strongly-chordal''
	graph. If a bound of $k$ could be shown for the boxicity of strongly-chordal
	graphs, one automatically shows a bound of $k+1$ for chordal bipartite graphs.
}	

Graphs which have no induced cycle of length at least 4 are called chordal
graphs. Chordal graphs in general can have unbounded boxicity since there
are split graphs (a subclass of chordal graphs) that have arbitrarily high
boxicity \cite{CozRob}. Strongly chordal graphs are chordal graphs
with no induced trampoline\cite{Farber83} (trampolines are also known as
``sun graphs''). Several other characterizations of
strongly chordal graphs can be found in \cite{McKee2003}, \cite{McKee2000},
\cite{Dahlhaus87} and \cite{Dahlhaus98}.
\ignore{
	Powers of strongly chordal graphs have also been shown to be strongly chordal
	though this is not true for chordal graphs.
}
\ignore{ -- for example,
the split graph obtained by completing one partition of the cocktail party
graph (the cocktail party graph.
}

\subsection{Leaf powers}
A graph $G$ is said to be a \emph{$k$-leaf power} if there exists a tree $T$
and a correspondence between the vertices of $G$ and the leaves of $T$ such
that two vertices in $G$ are adjacent if and only if the distance between
their corresponding leaves in $T$ is at most $k$. The tree $T$ is then called
a \emph{$k$-leaf root} of $G$. $k$-leaf powers were introduced by Nishimura
et. al.\cite{Nishimura} in relation to the phylogenetic reconstruction
problem in computational biology. Characterization of 3-leaf powers and a
linear time algorithm for their recognition was given in \cite{3leafpowers}.
Clearly, leaf powers are induced subgraphs of the powers of trees.
Now, since trees are strongly chordal and any power of any strongly chordal
graph is also strongly chordal (as shown in \cite{Raychauduri} and
\cite{Dahlhaus87}), leaf powers are also strongly chordal graphs.
\subsection{Our results}
We show that the boxicity of any $k$-leaf power is at most $k-1$ and also
demonstrate the tightness of this bound by constructing $k$-leaf powers
that have boxicity equal to $k-1$, for $k>1$. The tightness result implies
that strongly chordal graphs can have arbitrary boxicity. This is somewhat
surprising because when we study the boxicity of strongly chordal graphs,
it is tempting to conjecture that boxicity of any strongly chordal graph
may be bounded above by some constant and small examples seem to confirm this
conjecture.
A subclass of strongly chordal graphs, called \emph{strictly chordal graphs}, is
studied in \cite{Kennedy}. The graphs in this class are shown to be
4-leaf powers in \cite{klleaf}. Therefore strictly chordal graphs have boxicity
at most 3.
\section{Definitions and notations}
We study only simple, undirected and finite graphs.
Let $G(V,E)$ denote a graph $G$ on vertex set $V(G)$ and edge set $E(G)$.
For any graph
$G$, the number of edges in it is denoted by $||G||$. Thus, if $P$ is a path,
$||P||$ denotes the length of the path.
If $T$ is a tree that contains vertices $u$ and $v$, then $uTv$ denotes the
unique path in $T$.
For $u,v\in V(T)$, let $d_T(u,v):=||uTv||$ be the distance between $u$ and $v$
in $T$.
The $k$-th power of a graph $G$, denoted by $G^k$, is the graph
with vertex set $V(G^k)=V(G)$ and edge set $E(G^k)=\{(u,v)~|~u,v\in V(G)
\mbox{ and }d_G(u,v)\leq k\}$.

A set $X$ of three independent vertices in a graph $G$ is said to form an
\emph{asteroidal triple} if for any $u\in X$, there exists a path $P$ between
the two vertices in $X-\{u\}$ such that $N(u)\cap V(P)=\emptyset$ where $V(P)$
denotes the set of vertices in $P$. A graph is said to be \emph{asteroidal
triple-free}, or AT-free in short, if it does not contain any
asteroidal triple.

\begin{lemma}[Lekkerkerker and Boland\cite{Lekker}]\label{lekkerlem}
A graph is an interval graph if and only if it is chordal and asteroidal
triple-free.
\end{lemma}

If $G_1,\ldots,G_k$ are graphs on the same vertex set $V$, we
denote by $G_1\cap\cdots\cap G_k$ the graph on $V$ with edge set
$E(G_1)\cap\cdots\cap E(G_k)$.

\begin{lemma}[Roberts\cite{Roberts}]\label{robertslem}
For any graph $G$, $\boxi(G)\leq k$ if and only if there exists a collection
of $k$ interval graphs $I_1,\ldots,I_k$ such that $G=\bigcap_{i=1}^k I_i$.
\end{lemma}


A \emph{critical clique} in a graph is a maximal clique such that every vertex
in the
clique has the same neighbourhood in $G$.
The \emph{critical clique graph} of a graph
$G$,
denoted as $CC(G)$, is a graph in which there is a vertex for every critical
clique of $G$ and two vertices in $CC(G)$ are adjacent if and only if the
critical
cliques corresponding to them in $G$ together induce a clique in $G$.

\begin{lemma}\label{boxcclem}
For any graph $G$, $\boxi(G)=\boxi(CC(G))$.
\end{lemma}
\begin{proof}
Since $CC(G)$ is an induced subgraph of $G$, $\boxi(CC(G))\leq \boxi(G)$.
Now suppose that $u$ is a vertex in $G$ and $G'$ is the graph formed by adding
a vertex $u'$ to $V(G)$ such that $V(G')=V(G)\cup \{u'\}$ and $E(G')=E(G)\cup
(u,u')\cup\{(x,u')~|~(x,u)\in E(G)\}$. Since a $k$-box representation $f'$ for
$G'$ can be obtained from a $k$-box representation $f$ for $G$ by
extending $f$ to $f'$ by defining $f'(u)=f(u)$,
$\boxi(G')\leq \boxi(G)$. Now since any
graph $G$ can be obtained from $CC(G)$ by repeatedly performing this operation,
$\boxi(G)\leq \boxi(CC(G))$.
\hfill$\qed$
\end{proof}

A graph $G$ is a \emph{$k$-Steiner power} if there exists a tree $T$, called
the \emph{$k$-Steiner root} of $G$ with $|V(T)|\geq |V(G)|$, and an injective
map $f$ from  $V(G)$ to $V(T)$ such that for $u,v\in V(G)$, $(u,v)\in E(G)
\Leftrightarrow d_T(f(u),f(v))\leq k$. Note that $G$ is induced in $T^k$ by
the vertices in $f(V(G))$.

\begin{lemma}[Dom et al.\cite{Dom}]\label{domlem}
For $k\geq 3$, a graph $G$ is a $k$-leaf power if and only if $CC(G)$ is a
$(k-2)$-Steiner power.
\end{lemma}

We first study the boxicity of tree powers and then deduce our results for
leaf powers as corollaries.

\section{Boxicity of tree powers}
\subsection{An upper bound}
\par We show that if $T$ is any tree, boxicity of $T^k$ is at most $k+1$.

\medskip
Let $T$ be any tree.
Fix some non-leaf vertex $r$ to be the root of the tree.
Let $m$ be the number of leaves of the tree $T$. Let $l_1,\ldots,l_m$ be the
leaves of $T$ in the order in which they appear in some depth-first traversal
of $T$ starting from $r$.

Define the \emph{ancestor} relation on $V(T)$ as follows: a vertex
$u$ is said to be an ancestor of a vertex $v$, denoted as $u\preceq v$, if
$u\in rTv$. Similarly, we use the notation $u\succeq v$ to denote the fact that
$u$ is a \emph{descendant} of $v$, or in other words, $v$ is an ancestor of
$u$.

For any vertex $u\neq r$, let $p(u)$ be the \emph{parent} of $u$, i.e. the only
ancestor of $u$ adjacent to it. Let $p(r)=r$.
For any vertex $u$,
we define $p^0(u)=u$, $p^1(u)=p(u)$ and $p^i(u)=p(p^{i-1}(u))$, for $i\geq 2$.

For any vertex $u$, define $L(u)$ to be the set of indices
of leaves of $T$ that are descendants of $u$, i.e., $L(u)=\{i~|~l_i\succeq u\}$.
Define $s(u)=\min\{L(u)\}$ and $t(u)=\max\{L(u)\}$.

\ignore{
	\begin{lemma}\label{dfslemma}
	\begin{enumerate}
	\item If $u\preceq v$ then $L(v)\subseteq L(u)$.
	item If $u\not\preceq v$ and $v\not\preceq u$ then $L(u)\cap L(v)=\emptyset$.
	\end{enumerate}
	\end{lemma}
	Clear from the fact that the leaves were numbered in the order in which they
	appeared in a depth-first traversal of $T$ from $r$.\hfill$\qed$\\
}
\begin{lemma}\label{stlemma}
If $u\preceq v$, then $s(u)\leq s(v)\leq t(v)\leq t(u)$.
\end{lemma}
\begin{proof}
$u\preceq v\Rightarrow L(v)\subseteq L(u)$. Hence the lemma follows.
\hfill$\qed$
\end{proof}

\begin{lemma}\label{stlemma1}
If $u\not\preceq v$ and $v\not\preceq u$, then either $s(u)\leq t(u)<s(v)$ or
$s(v)\leq t(v)<s(u)$.
\end{lemma}
\begin{proof}
Since the leaves were ordered in the sequence in which they appear in a
depth-first traversal of $T$ from $r$, for any vertex $u$, the leaves in
$L(u)$ appear consecutively in the ordering $l_1,\ldots,l_m$.
Since $u\not\preceq v$ and $v\not\preceq u$, $L(u)\cap L(v)=\emptyset$.
This proves the lemma.\hfill$\qed$
\end{proof}

\medskip
In order to show that $\boxi(T^k)\leq k+1$, we construct $k+1$ interval graphs
$I',I_0,\ldots,I_{k-1}$ such that $T^k=I'\cap I_0\cap\cdots\cap I_{k-1}$. These
interval graphs are constructed as follows.\\

\noindent\textbf{Construction of $I_i$, $0\leq i \leq k-1$:}\\
Let $f_i(u)$ be the interval assigned to vertex $u$ in $I_i$, i.e.,
$V(I_i)=V(T)$ and $E(I_i)=\{(u,v)~|~f_i(u)\cap f_i(v)\not=\emptyset\}$.
$f_i$ is defined as: $$f_i(u)=[s(p^i(u)),t(p^{k-1-i}(u))]$$
Note that from Lemma \ref{stlemma}, $s(p^i(u))\leq t(p^{k-1-i}(u))$ since
either $p^i(u)\preceq p^{k-1-i}(u)$ or $p^{k-1-i}(u)\preceq p^i(u)$.
Therefore $f_i(u)$ is always a valid closed interval on the real line.

\medskip
\noindent\textbf{Construction of $I'$:}\\
$V(I')=V(T)$ and $E(I')=\{(u,v)~|~f'(u)\cap f'(v)\not=\emptyset\}$
where $f'$ is defined as: $$f'(u)=[d_T(r,u),d_T(r,u)+k]$$

\medskip
\begin{lemma}\label{Iisupgraph}
For $0\leq i\leq k-1$, $I_i$ is a supergraph of $T^k$.
\end{lemma}
\begin{proof}
Let $(u,v)\in E(T^k)$. We will show that $(u,v)\in E(I_i)$.
Let $P$ be the path between $u$ and $v$ in $T$.
Since $(u,v)\in E(T^k)$, $||P||\leq k$. It is easy to see that there is
exactly one vertex $x$ on $P$
such that $x\preceq u$ and $x\preceq v$. Note that $x$ is the least common
ancestor of $u$ and $v$. Let $d_1=||uPx||$ and $d_2=||vPx||$.
Thus, $x=p^{d_1}(u)=p^{d_2}(v)$ and $||P||=d_1+d_2\leq k$.

Let us assume without loss of generality that $s(p^i(u))\leq s(p^i(v))$.

If $i\geq d_2$,
then $p^i(v)\preceq x \preceq u$ and by Lemma \ref{stlemma}, $s(p^i(v))\leq t(u)$ and
also by Lemma \ref{stlemma}, $t(u)\leq t(p^{k-1-i}(u))$ implying that
$s(p^i(v))\leq t(p^{k-1-i}(u))$. We now have $s(p^i(u))\leq s(p^i(v))\leq
t(p^{k-1-i}(u))$. Thus, $f_i(u)\cap f_i(v)\not=\emptyset$ and therefore,
$(u,v)\in E(I_i)$.

Now, if $i<d_2$,
we have $k-1-i\geq d_1$. Therefore, $p^{k-1-i}(u)\preceq x \preceq v$ and by Lemma
\ref{stlemma}, $t(v)\leq t(p^{k-1-i}(u))$ and again by Lemma \ref{stlemma},
$s(p^i(v))\leq t(v)$ and so we have $s(p^i(v))\leq t(p^{k-1-i}(u))$. This
means that $s(p^i(u))\leq s(p^i(v))\leq t(p^{k-1-i}(u))$. Thus, $f_i(u)
\cap f_i(v)\not=\emptyset$ implying that $(u,v)\in E(I_i)$.\hfill$\qed$
\end{proof}

\begin{lemma}\label{I'supgraph}
$I'$ is a supergraph of $T^k$.
\end{lemma}
\begin{proof}
Let $(u,v)\in E(T^k)$. We have to show that $(u,v)\in E(I')$.
Let $P=uTv$ and let $x$ be the vertex on $P$ such that $x\preceq u$
and $x\preceq v$ (i.e., $x$ is the least common ancestor of $u$ and $v$).
Let $d_1=||uPx||$, $d_2=||vPx||$ and $d_3=||rTx||$. We have
$d_T(r,u)=d_3+d_1$ and $d_T(r,v)=d_3+d_2$. Also, since $(u,v)\in E(T^k)$,
$d_1+d_2\leq k$.
Therefore, $|d_1-d_2|\leq k$ which means that $|d_T(r,u)-d_T(r,v)|\leq k$.
Thus, we have $f'(u)\cap f'(v)\not=\emptyset$ implying that $(u,v)\in E(I')$.
\hfill$\qed$
\end{proof}

\begin{lemma}\label{nonedges}
If $(u,v)\not\in E(T^k)$, then either $(u,v)\not\in E(I')$ or
$\exists i$ such that $(u,v)\not\in E(I_i)$.
\end{lemma}
\begin{proof}
Let $(u,v)\not\in E(T^k)$. Let $P=uTv$ and again let $x$ be the least common
ancestor of $u$ and $v$, i.e., $x$ is the vertex on $P$ such
that $x\preceq u$ and $x\preceq v$. Define $d_1=||uPx||$ and $d_2=||vPx||$;
thus,
$x=p^{d_1}(u)=p^{d_2}(v)$. Since $(u,v)\not\in E(T^k)$, we have $d_1+d_2>k$.

\begin{case}
$d_1\not=0$ and $d_2\not=0$.
\end{case}
Let us assume without loss of generality that $s(p^{d_1-1}(u))\leq s(p^{d_2-1}
(v))$
By the definition of $d_1$ and $d_2$, we have $p^{d_1-1}(u)\not\preceq
p^{d_2-1}(v)$ and $p^{d_2-1}(v)\not\preceq p^{d_1-1}(u)$).
Then by Lemma \ref{stlemma1}, $t(p^{d_1-1}(u))<s(p^{d_2-1}(v))$.
Now applying Lemma \ref{stlemma}, we get
\begin{equation}\label{eqn}
\mbox{for any }i,j\mbox{ such that }0\leq i<d_1,0\leq j<d_2, t(p^i(u))<s(p^j(v))
\end{equation} 

If $1\leq d_2\leq k$, consider the interval graph $I_j$ where $j=d_2-1$.
Now, let $i=k-1-j=k-d_2<d_1$. Now, from (\ref{eqn}), we get
$t(p^i(u))<s(p^j(v))$, that is to say $t(p^{k-1-j}(u))<s(p^j(v))$.
Thus, $f_j(u)\cap f_j(v)=\emptyset$ which means that $(u,v)\not\in E(I_j)$.

If $d_2>k$, then consider $I_{k-1}$. From (\ref{eqn}), we have $t(p^0(u))<
s(p^{k-1}(v))$, and therefore $f_{k-1}(u) \cap f_{k-1}(v)=\emptyset$.
Thus, $(u,v)\not\in E(I_{k-1})$.

\begin{case}
$d_1=0$ or $d_2=0$.
\end{case}
Now, if $d_1=0$, then $u=x\preceq v$ and $d_2>k$. This implies that
$d_T(r,v)>d_T(r,u)+k$.
Similarly, if $d_2=0$, then $v=x\preceq u$ and $d_1>k$ implying that
$d_T(r,u)>d_T(r,v)+k$. In either case,
$f'(u)\cap f'(v)=\emptyset$, and so $(u,v)\not\in E(I')$.
\hfill$\qed$
\end{proof}

\begin{theorem}
For any tree $T$, $\boxi(T^k)\leq k+1$, for $k\geq 1$.
\end{theorem}
\begin{proof}
Let $I',I_0,\ldots,I_{k-1}$ be the interval graphs constructed as
explained above.
Lemmas \ref{Iisupgraph}, \ref{I'supgraph} and \ref{nonedges} suffice to show
that $T^k=I'\cap I_0\cap\cdots\cap I_{k-1}$. Thus, by Lemma \ref{robertslem},
we have the theorem.
\hfill\bbox
\end{proof}

\begin{corollary}\label{upboundcor}
If $G$ is a $k$-leaf power, $\boxi(G)\leq k-1$, for $k\geq 2$.
\end{corollary}
\begin{proof}
It is easy to see that 2-leaf powers are collections of disjoint cliques
and thus have boxicity 1. Thus, the corollary is true for $k=2$. For
$k\geq 3$, the statement of the corollary can be proved as follows.
From Lemma \ref{boxcclem}, we have $\boxi(G)=\boxi(CC(G))$.
From Lemma \ref{domlem}, $CC(G)$ has a
$(k-2)$-Steiner root, say $T$. Now, it follows that $\boxi(G)=\boxi(CC(G))
\leq \boxi(T^{k-2})\leq k-1$.

\hfill\bbox
\end{proof}

\subsection{Tightness of the bound}
Let the function $w: \mathbb{Z}^{+} \rightarrow \mathbb{Z}^{+}$ be defined recursively as 
follows: 
\\ $w(1)  =  1$, $w(2) = 3$ and for any $i \geq 3$,  
$$w(i)  =  2(i-1) + 1 + \left[{i-1 \choose 2} \cdot 4\cdot(w(i-2) -1) + 1\right] $$
For any $k \in \mathbb{N}$ and $k \geq 1$, let $S_k$ be the tree shown in figure 
\ref{SkFigure}.

\begin{figure}[h]
\begin{center}
\input{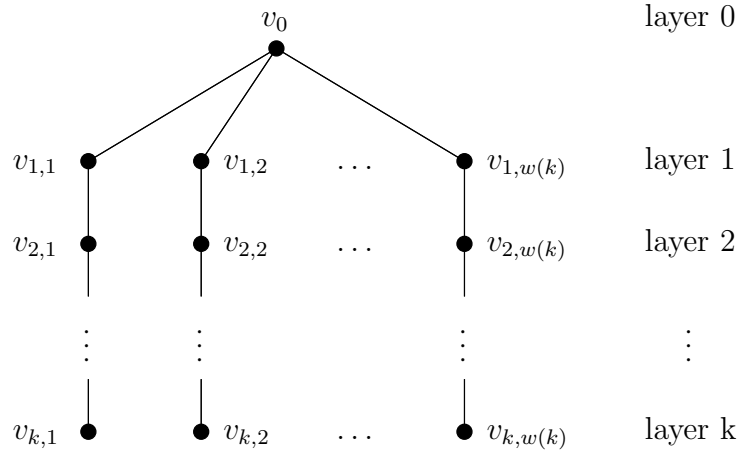}
\caption{Tree $S_k$}
\label{SkFigure}
\end{center}
\end{figure}

\begin{lemma}
\label{SkLemma}
$\boxi\left({(S_k)}^k\right) > k-1$.
\end{lemma}
\begin{proof}
Let us prove this using induction on $k$. It is easy to see that 
$\boxi \left((S_1)^1 \right) > 0$ and $\boxi \left((S_2)^2\right) > 1$
(in $(S_2)^2$ vertices $v_{2,1}, v_{2,2}$ and $v_{2,3}$ form an asteroidal
triple and therefore by Lemma \ref{lekkerlem}, $(S_2)^2$ is not an interval graph).
Let $m \geq 3$ be a positive integer
and assume that the statement of the lemma is true for any $k \leq m-1$.
We shall now prove by contradiction that $\boxi\left((S_m)^m\right) > m-1$.
For ease of notation, let $S = S_m$.
If $\boxi(S^m)\leq m-1$, then by Lemma \ref{robertslem}, there exist
$m-1$ interval graphs $I_1, I_2, \ldots , I_{m-1}$
such that
$S^m = I_1 \cap \cdots \cap I_{m-1}$. Let $\mathcal{I} = \{I_1, I_2, \ldots,
I_{m-1}\}$. For each interval graph $I_p$, choose an interval
representation $\mathcal{R}_p$.
For any $u \in V(S^m)$ and $I_p \in \mathcal{I}$, let 
$left(u,I_p)$ ($right(u,I_p)$) denote the left (right) endpoint of its 
interval in $\mathcal{R}_p$. 
We define $L_i = \{v_{i,1}, v_{i,2}, \ldots ,v_{i,w(m)}\}$, i.e. the set of
all vertices in the $i$-th layer of $S^m$.
Let $interval(u,I_p)$ denote $[left(u,I_p), right(u,I_p)]$, the interval
that corresponds to the vertex $u$ in $\mathcal{R}_p$.
Note that, since $m \geq 3$, the vertices in layer $1$ of $S^m$ form a clique. 
Therefore, by Helly property, in the interval representation $\mathcal{R}_p$ 
of each interval graph $I_p$, the intervals corresponding to the vertices of
layer 1 have a common intersection region. Let $y_p$ and $z_p$ denote the 
left and right endpoints respectively of this common intersection region in
$\mathcal{R}_p$.
That is, $[y_p, z_p] = \bigcap_{j=1}^{w(m)} interval(v_{1,j}, I_p)$.

Since a vertex in $L_m$, say $v_{m,j}$, is not adjacent to any vertex $v_{1,j'}$
in layer 1, for $j' \neq j$, there exists at least one interval graph $I_p$ such that 
$interval(v_{m,j},I_p)$ is disjoint from the abovementioned common intersection
region $[y_p,z_p]$. Define $F(v_{m,j}) = \{I_p\in\mathcal{I} ~|~ interval(v_{m,j},
I_p) \cap
[y_p, z_p] = \emptyset \}$, i.e., the collection of all interval graphs in which
$v_{m,j}$ is not adjacent to at least one vertex in layer 1.

Also define $Q(I_p)=\{v_{m,j}\in L_m ~|~ I_p\in F(v_{m,j})\}$, i.e., the set of
all vertices in layer $m$ whose intervals are disjoint from $[y_p,z_p]$ in
$\mathcal{R}_p$.
Let us partition $Q(I_p)$ into two sets $Q_l(I_p)$ and $Q_r(I_p)$.
$$Q_l(I_p)=\{v_{m,j}\in Q(I_p)~|~left(v_{m,j},I_p) \leq right(v_{m,j},I_p) < y_p
\leq z_p\}$$
$$Q_r(I_p)=\{v_{m,j}\in Q(I_p)~|~y_p \leq z_p<left(v_{m,j},I_p) \leq
right(v_{m,j},I_p)\}$$


Partition $L_m$ into two sets $A$ and $B$ such that $A=\{v_{m,j}~|~
|F(v_{m,j})|=1\}$
and $B=\{v_{m,j}~|~|F(v_{m,j})|>1\}$. Since $|A|+|B|=|L_m|=w(m)=2(m-1)+1+\left[{m-1
\choose 2} \cdot 4 \cdot(w(m-2) -1) + 1\right]$, we encounter at least one of
the following two cases. We will show that both the cases lead to contradictions.
\setcounter{case}{0}
\begin{case}
$|A|\geq 2(m-1)+1$.
\end{case}
Let us partition $A$ into sets $A_1,A_2,\ldots,A_{m-1}$ where $A_i=\{u\in A~|~
F(u)=\{I_i\}\}$. Since $|A|\geq 2(m-1)+1$, there exists an $A_p$ with
$|A_p|\geq 3$. For a vertex $u\in A_p$, $interval(u,I_p)$ can be either to the left
or to the right of $[y_p,z_p]$ in $\mathcal{R}_p$. Thus $A_p$ can be further
partitioned into $A_p^l$ and $A_p^r$ where
$A_p^l=A_p\cap Q_l(I_p)$
 and
$A_p^r=A_p\cap Q_r(I_p)$.
Since $|A_p|\geq 3$, we have $|A_p^l|\geq 2 $ or $|A_p^r|\geq 2$. Without
loss of generality, let $|A_p^l|\geq 2$ with $v_{m,j},v_{m,j'}\in A_p^l$.
Also assume without loss of generality that $right(v_{m,j},I_p)\leq
right(v_{m,j'},I_p)<y_p$. Since $v_{1,j}$ is adjacent to $v_{m,j}$, we have
$interval(v_{1,j},I_p)\cap interval(v_{m,j},I_p)\neq \emptyset$.
Also, by the definition of $[y_p,z_p]$, $interval(v_{1,j},I_p)\cap [y_p,z_p]
\neq \emptyset$. Therefore, $interval(v_{1,j},I_p)$ contains both the points
$right(v_{m,j},I_p)$ and $y_p$, implying that it also contains $right(v_{m,j'},
I_p)$. Thus,
$(v_{1,j},v_{m,j'})\in E(I_p)$.
Since $F(v_{m,j'})=\{I_p\}$, we know that for all $p'\neq p$, $interval(v_{m,j'}
,I_{p'})\cap [y_{p'},z_{p'}]\neq\emptyset$ and therefore $(v_{1,j},v_{m,j'})
\in E(I_{p'})$.
This implies that $(v_{1,j},v_{m,j'})\in E(I_1\cap\ldots\cap I_{m-1})$,
a contradiction.

\begin{case}
$|B| \geq \left[{m-1 \choose 2} \cdot 4\cdot(w(m-2) -1)\right] + 1$.
\end{case}

For $u\in B$, let $g(u)=\min_{I_i\in F(u)}\{i\}$ and let $g'(u)=\min_{I_i\in
F(u)-\{I_{g(u)}\}}\{i\}$. Define $X(u)=\{g(u),g'(u)\}$. Note that both $g(u)$
and $g'(u)$ exists since $u\in B$ and thus $|F(u)|\geq 2$.
Let $B_{ij}=\{u\in B~|~X(u)=\{i,j\}\}$. Thus $\mathcal{P}=\{B_{ij}~|~\{i,j\}
\subseteq \{1,\ldots,m-1\}\}$ is a partition of $B$ into $m-1 \choose 2$ sets.
Since $|B| \geq \left[{m-1 \choose 2} \cdot 4\cdot(w(m-2) -1)\right] + 1$,
there exists $B_{pq}\in \mathcal{P}$ such that $|B_{pq}|\geq 4\cdot(w(m-2)
-1)+1$. Now we partition $B_{pq}$ into 4 sets namely,
\begin{center}
$B_{pq}^{ll}=B_{pq}\cap Q_l(I_p)\cap Q_l(I_q)$\\
$B_{pq}^{lr}=B_{pq}\cap Q_l(I_p)\cap Q_r(I_q)$\\
$B_{pq}^{rl}=B_{pq}\cap Q_r(I_p)\cap Q_l(I_q)$\\
$B_{pq}^{rr}=B_{pq}\cap Q_r(I_p)\cap Q_r(I_q)$
\end{center}
Since $|B_{pq}|\geq 4\cdot(w(m-2)-1)+1$, one of these 4 sets
will have cardinality at least $w(m-2)$. Let this set be $B_{pq}^{lr}$ (the
proof is similar for all the other cases). Thus $B_{pq}^{lr}$ contains $w(m-2)$
vertices, which we will assume without loss of generality to be $v_{m,1},\ldots,
v_{m,w(m-2)}$. Note that for any $v_{m,j}\in B_{pq}^{lr}$, $right(v_{m,j},I_p)
<y_p$ and $z_q<left(v_{m,j},I_q)$.
Let $Y=\{v_{i,j}~|~2\leq i\leq m-1, 1\leq j\leq w(m-2)\}$. Now, since in $I_p$
any vertex $v_{i,j}$ in $Y$ is adjacent to both $v_{m,j}$ and to all the vertices
of layer 1, we have $interval(v_{i,j},I_p)\cap
interval(v_{m,j},I_p)\neq\emptyset$ and $interval(v_{i,j},I_p)\cap
[y_p,z_p]\neq\emptyset$. Since $right(v_{m,j},I_p)<y_p$,
$interval(v_{i,j},I_p)$ contains
the point $y_p$. Similarly, $interval(v_{i,j},I_q)$ contains
the point $z_q$. Thus, $Y$ induces a clique in
both $I_p$ and $I_q$. Since $v_0$ is a universal vertex in $S^m$, $\{v_0\}\cup
Y$ also induces a clique in both $I_p$ and $I_q$.
We claim that in $S^m$, the
induced subgraph of $\{v_0\}\cup Y$ is isomorphic to $(S_{m-2})^{m-2}$.
To see this, let
$V((S_{m-2})^{m-2})=\{\overline{v}_0,\overline{v}_{1,1},\ldots,\overline{
v}_{1,w(m-2)},\overline{v}_{2,1},\ldots,\overline{v}_{2,w(m-2)},\ldots,
\overline{v}_{m-2,1},\ldots,\overline{v}_{m-2,w(m-2)}\}$.
The isomorphism is given by the bijection $f:\{v_0\}\cup Y\rightarrow
V((S_{m-2})^{m-2})$ where $f(v_0)=\overline{v}_0$ and $f(v_{i,j})=
\overline{v}_{i-1,j}$. It can be easily verified that $f$ is an isomorphism
from the graph induced in $S^m$ by $\{v_0\}\cup Y$ to $(S_{m-2})^{m-2}$.
Let
$$G'=\bigcap_{I_i \in \mathcal{I}\setminus \{I_p, I_q\}} I_i$$
Since $\{v_0\}\cup Y$ induced a clique in $I_p$ and $I_q$, the induced
subgraph on $\{v_0\}\cup Y$ in $G'$ is the same as the induced subgraph
on $\{v_0\}\cup Y$ in $S^m$, i.e., $(S_{m-2})^{m-2}$ is an induced subgraph
of $G'$. Therefore,
$\boxi((S_{m-2})^{m-2})\leq \boxi(G')\leq m-3$ (from Lemma \ref{robertslem}).
But this contradicts the induction hypothesis.\hfill$\qed$
\end{proof}

\begin{figure}[h]
\begin{center}
\input{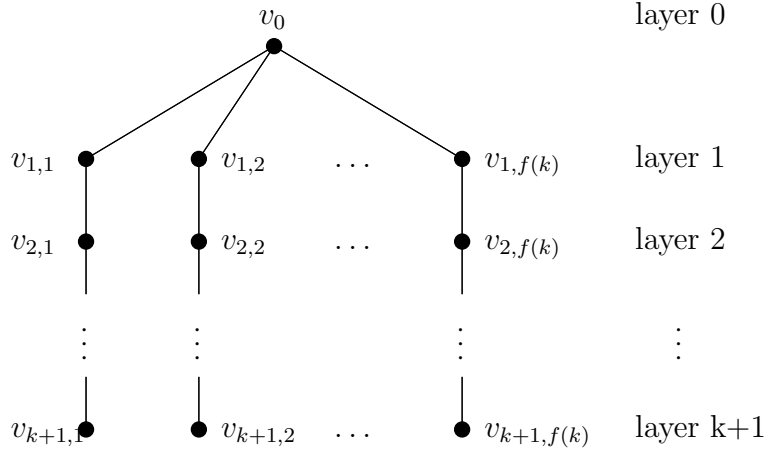}
\caption{Tree $T_k$}
\label{TkFigure}
\end{center}
\end{figure}

\medskip
We now construct a tree $T_k$ (see figure~\ref{TkFigure}), for any
$k \in \mathbb{N}$ and $k\geq 1$. Define $f(k) = 2k\cdot (w(k)-1) + 1$.

\begin{lemma}
\label{TkLemma}
$\boxi\left((T_k)^k\right) > k$.
\end{lemma}
\begin{proof}
We prove this by contradiction. Again, for ease of notation, let $T = T_k$.
Assume that $\boxi(T^k)\leq k$.
By Lemma \ref{robertslem}, there exists a collection of $k$ interval graphs
$\mathcal{I} = \{I_1, I_2, \ldots , I_k \}$ such that
$T^k=\bigcap_{I\in\mathcal{I}} I$.
Now for each interval graph $I_p$, for $1\leq p\leq k$, choose an interval
representation $\mathcal{R}_p$.
For a vertex $u\in V(T^k)$, let $left(u,I_p)$ ($right(u,I_p)$) denote left(right)
endpoint of its interval in $\mathcal{R}_p$.
Let $L_i = \{v_{i,1}, v_{i,2}, \ldots , v_{i,f(k)} \}$ be the set of all vertices
in the $i$-th layer of $T$.

For each vertex $v_{k+1,j}\in L_{k+1}$, since $(v_{k+1,j},v_0)\not\in E(T^k)$,
there exists at least one interval graph $I_p$ in which $interval(v_{k+1,j},
I_p)\cap interval(v_0,I_p)=\emptyset$. For each interval graph $Ip$,
we define $Q(I_p)=\{v_{k+1,j}\in L_{k+1}~|~interval(v_{k+1,j},I_p)\cap
interval(v_0,I_p)=\emptyset \mbox{ and } v_{k+1,j}\not\in
Q(I_{p'})\mbox{ for any }p'<p\}$. Note that $\{Q(I_1),\ldots,Q(I_k)\}$ is a
partition of $L_{k+1}$. We define a partition of $Q(I_p)$ into two sets
$Q_l(I_p)$ and $Q_r(I_p)$ as follows. For any vertex $u\in Q(I_p)$,
$u$ is in $Q_l(I_p)$
if the interval corresponding to $u$ is to the left of
the interval corresponding to $v_0$ in $\mathcal{R}_p$, otherwise it is in
$Q_r(I_p)$. That is,
$$Q_l(I_p)=\{u\in Q(I_p)~|~left(u, I_p) \leq right(u, I_p) < left(v_0, I_p)
\leq right(v_0, I_p)\}$$
$$Q_r(I_p)=\{u\in Q(I_p)~|~left(v_0, I_p) \leq right(v_0, I_p) < left(u, I_p)
\leq right(u, I_p)\}$$

Now,
$\{Q_l(I_i),Q_r(I_i)~|~1\leq i\leq k\}$ is a partition of $L_{k+1}$ into
$2k$ sets. Since $|L_{k+1}|=f(k)=
2k\cdot (w(k)-1) + 1$, there exists some set in this partition with size at least
$w(k)$. Let us assume this set to be $Q_l(I_p)$ for some $p$. The proof
is similar if the set is $Q_r(I_p)$ and therefore will not be
detailed here. Now, we have $|Q_l(I_p)|\geq w(k)$. Let us assume without loss
of generality that $v_{k+1,1},v_{k+1,2},\ldots,v_{k+1,w(k)}\in Q_l(I_p)$.
Let $Y=\{v_{i,j}~|~1\leq i\leq k\mbox{ and }1\leq j\leq w(k)\}$. Note that
any $v_{i,j}\in Y$ is adjacent to vertices $v_{k+1,j}$ and $v_0$
in $T^k$ and therefore also in $I_p$. Thus, $interval(v_{i,j},I_p)\cap
interval(v_{k+1,j},I_p)\neq\emptyset$ and $interval(v_{i,j},I_p)\cap
interval(v_0,I_p)\neq\emptyset$. Now, from the definition of $Q_l(I_p)$,
it is easy to see that $left(v_0,I_p)\in interval(v_{i,j},I_p)$ for any
$v_{i,j}\in Y$. This means that the vertices in $\{v_0\}\cup Y$ induce
a clique in $I_p$.

It is easy to see that in $T^k$, the subgraph induced by $\{v_0\}\cup Y$ is
isomorphic to ${(S_k)}^k$. Let
$$G' = \bigcap_{I_i \in \mathcal{I}\setminus \{I_p\}} I_i $$
Since the induced subgraph on $\{v_0\}\cup Y$ in $I_p$ is a clique,
the subgraph induced by $\{v_0\}\cup Y$ in $G'$ is the same as the subgraph
induced by $\{v_0\}\cup Y$ in $T^k$, i.e., $(S_k)^k$ is an induced
subgraph of $G'$. Therefore,
$\boxi((S_k)^k) \leq \boxi (G') \leq k-1$ (from Lemma \ref{robertslem}).
But this contradicts Lemma \ref{SkLemma}.\hfill$\qed$
\end{proof}

Hence we have the following theorem.

\begin{theorem}
For every $k \in \mathbb{N}$ and $k\geq 1$, $\exists$ a tree $\tau$ such that
$\boxi(\tau^k) > k$.
\end{theorem}

\begin{corollary}
For every $k \in \mathbb{N}$ and $k\geq 2$, $\exists$ a $k$-leaf power $G$ such
that $\boxi(G)=k-1$.
\end{corollary}
\begin{proof}
For $k=2$, any $k$-leaf power is a collection of disjoint cliques and thus
has boxicity 1. The proof for the case when $k\geq 3$ is as follows.
Let $G=(T_{k-2})^{k-2}$. Therefore, $G$ is a $(k-2)$-Steiner
power (in fact $T_{k-2}$ is a Steiner root for $G$ with no Steiner vertices).
Since $CC(G)$ and $G$ are the same graph (note that no two vertices in $G$ have
the same neighbourhood), from Lemma \ref{domlem},
$G$ is a $k$-leaf power. Now, Lemma \ref{TkLemma}
implies that $\boxi(G)>k-2$. Using corollary \ref{upboundcor}, we have
$\boxi(G)=k-1$.\hfill\bbox

\end{proof}

\end{document}